\documentclass[sn-basic]{sn-jnl}

\usepackage{graphicx}%
\usepackage{multirow}%
\usepackage{amsmath,amssymb,amsfonts}%
\usepackage{amsthm}%
\usepackage{mathrsfs}%
\usepackage[title]{appendix}%
\usepackage{xcolor}%
\usepackage{textcomp}%
\usepackage{manyfoot}%
\usepackage{booktabs}%
\usepackage{algorithm}%
\usepackage{algorithmicx}%
\usepackage{algpseudocode}%
\usepackage{listings}%
\usepackage[pagewise]{lineno}

\usepackage{amsmath}
\usepackage{twoopt}
\usepackage{tikz-cd}
\usepackage{pb-diagram}
\usetikzlibrary{matrix,arrows,decorations.pathmorphing,decorations.markings} 

\newcommandtwoopt\eck[3][G][\Bbbk]{H^*_{#1}(#3; #2)}

\newcommand\ec[2][G]{H^*_{#1}(#2)}
\newcommandtwoopt\csk[2][G][\Bbbk]{H^*(B{#1}; #2)}
\newcommand\cs[1][G]{H^*(B#1)}

\newcommand{\Z}{\mathbb Z}
\newcommand{\F}{\mathbb F}
\newcommand{\R}{\mathbb R}
\newcommand{\C}{\mathbb C}

\newcommand{\Q}{\mathbb Q}

\newcommand{\Zd}{\mathbb{Z}/2\mathbb{Z}}

\newcommand{\bk}{\Bbbk}

\DeclareMathOperator{\coker}{coker}

\newtheorem{theorem}{Theorem}[section]
\newtheorem{proposition}[theorem]{Proposition}%
\newtheorem{corollary}[theorem]{Corollary}
\newtheorem{lemma}[theorem]{Lemma}

\newtheorem{example}[theorem]{Example}%
\newtheorem{remark}[theorem]{Remark}%

\newtheorem{definition}[theorem]{Definition}%

\begin{document}

\title[Equivariant cohomology for semidirect product actions]{Equivariant cohomology for semidirect product actions}


\author{\fnm{Sergio} \sur{Chaves}}\email{schavesr@math.rochester.edu}



\affil{\orgdiv{Department of Mathematics}, \orgname{University of Rochester}, \state{NY}, \country{USA}}




\abstract{	The rational Borel equivariant cohomology for actions of a compact connected Lie group is determined by restriction of the action to a maximal torus. We show that a similar reduction holds for any compact Lie group $G$ when there is a closed subgroup $K$ such that the cohomology of the classifying space $BK$ is free over the cohomology of $BG$ with field coefficients. This provides a different approach to the equivariant cohomology of a space with a torus action and a compatible involution, and we relate this description with results for 2-torus actions.}




\maketitle

\section{Introduction}
Let $G$ be a compact Lie group and $X$ be a finite $G$-CW complex. The $G$-equivariant cohomology of $X$ with coefficients over a field $\Bbbk$ is defined as the singular cohomology of the homotopy quotient $X_G := EG \times_G X$; namely, $H^*_G(X;\Bbbk):= H^*( X_G; \Bbbk)$. It becomes canonically a module over the cohomology of the classifying space $H^*(BG; \Bbbk)$. We will omit cohomology coefficients as long as there is no ambiguity.  We say that $X$ is $G$-equivariantly formal over $\Bbbk$ if the restriction map $H^*_G(X) \rightarrow H^*(X)$ is surjective. In this situation, the Leray-Hirsch theorem implies that $H^*_G(X)$ is a free module over $H^*(BG)$.

Freeness of equivariant cohomology has been generalized to the study of syzygy modules. A detailed discussion of this topic was started by Allday-Franz-Puppe  \cite{allday2014equivariant} for torus actions  over a field of characteristic zero. Recall that a finitely generated module $M$ over a commutative ring $R$ is a $j$-th syzygy if there is an exact sequence

\begin{equation}\label{eq:syzygy}
0 \rightarrow M \rightarrow F_1 \rightarrow \cdots \rightarrow F_j
\end{equation}

of free $R$-modules $F_k$ for $1 \leq k \leq j$. If $R$ is a polynomial algebra in $n$ variables over a field $\bk$, then the $n$-th syzygy modules are free as a consequence of Hilbert Syzygy theorem.

The study of syzygies in equivariant cohomology for torus actions can be extended to actions of any compact connected Lie group actions by considering the restriction of the action to a maximal torus \cite{franz2016syzygies}, and elementary $p$-abelian groups actions by restriction and transfer of the action to an associated torus action \cite{ptori}.

In this paper, we study the more general problem of characterizing syzygies in equivariant cohomology for compact Lie group actions in terms of actions of suitable closed subgroups. In particular, we allow the acting group to be disconnected, as this approach includes interesting actions of finite groups and orthogonal matrices among others. This study generalizes the relationship between compact connected Lie groups and their maximal tori in equivariant cohomology as discussed before.

Let us consider a compact Lie group $G$ and let $K \subseteq G$ be a closed subgroup. We denote by $W = N_G(K)/C_G(K)$ the Weyl group of $K$ in $G$. We also consider cohomology with coefficients over a field $\bk$. Suppose that the homogeneous space $G/K$ is connected, that the canonical map $H^*(BK) \rightarrow H^*(G/K)$ arising from the fibration $G/K \rightarrow BK \rightarrow BG$  is surjective  and that there is an isomorphism of algebras $H^*(BK)^W \cong H^*(BG)$. One of the main results of this paper is the following.

\begin{theorem}\label{teo1.1}
	Let $G$ be a compact Lie group that admits a closed subgroup $K$ satisfying the conditions above. Let $X$ be a $G$-space such that $H^*(X)^G = H^*(X)$. Then $W$ acts on the $K$-equivariant cohomology of $X$, there is a natural isomorphism of $H^*(BG)$-algebras $H^*_G(X) \cong H^*_K(X)^W$
	and a natural isomorphism of $H^*(BK)$-algebras $H^*_K(X) \cong H_G^*(X) \otimes_{H^*(BG)} H^*(BK)$.
\end{theorem}
Observe that surjectivity of the map $H^*(BK) \rightarrow H^*(G/K)$ allows us to describe the cohomology of the homogeneous space $G/K$ in terms of the cohomology of the classifying spaces of $G$ and $K$. This applies to the cohomology of homogeneous spaces of Lie groups \cite{baum1968cohomology}, \cite{may1968cohomology},  and  the equivariant cohomology of Hamiltonian actions of non-abelian compact connected Lie groups in symplectic geometry \cite{baird2018cohomology}. 
As a consequence of Theorem \ref{teo1.1}, we can characterize syzygies in $G$-equivariant cohomology in terms of the $K$-equivariant cohomology.

\begin{corollary}
	The module $\ec{X}$ is a $j$-th syzygy over $\cs$ if and only if $\ec[K]{X}$ is a $j$-th syzygy over $\cs[K]$.
\end{corollary}

With our methods, we recover in equivariant cohomology the classical results for compact connected Lie groups with cohomology over rational coefficients. They also allow us to study the equivariant cohomology for actions of semidirect product of groups. For example, actions of some matrix groups, and torus action with compatible antisymplectic involutions. In particular, the latter case has been of interest in symplectic geometry. Namely, let $M$ be a symplectic manifold with a symplectic action of a torus $T$ and an antisymplectic compatible involution $\tau$. The maximal elementary $2$-abelian subgroup $T_2$ of $T$ acts on the subspace of fixed points $M^\tau$, and if $M$ is $T$-equivariantly formal over $\Q$, then $M^\tau$ is $T_2$-equivariantly formal over $\F_2$. See \cite{A}, \cite{D}, \cite{Fr},

The latter situation motivates our study of  equivariant cohomology for semidirect product actions; in particular, we use it to approach the symplectic setting described above by considering the equivariant cohomology for actions of the group $T \rtimes \Zd$. 

This document is organized as follows: In Section \ref{se:FEP}, we discuss free extension pairs and the reduction of syzygies in equivariant cohomology for a pair of groups satisfying this property. In Section \ref{se:toruscomp}, we approach torus actions and compatible involutions by looking at the induced action of the semidirect product of a torus and a 2-tori. In Section \ref{se:2torus},  we explore the reduction of syzygies in equivariant cohomology for torus actions with compatible invlution to 2-torus actions. Finally we discuss a topological generalization of Hamiltonian actions on a symplectic manifold with an antisymplectic compatible involution using the results discussed in this document.

\section{Free extension pairs}\label{se:FEP}

Let $G$ be a compact connected Lie group and $T$ be a maximal torus in $G$.  Since the module $H^*(BT; \Q)$ is free over $H^*(BG;\Q)$ via the map induced by the inclusion, the rational equivariant cohomology of a $G$-space $X$ is completely determined by the restricted action of $T$ on $X$. There is an isomorphism of $H^*(BT;\Q)$-algebras

\begin{equation}\label{eq:maxtorus}
H^*_T(X;\Q) \cong H^*_G(X;\Q) \otimes_{H^*(BG;\Q)} H^*(BT;\Q)
\end{equation}
and an isomorphism of $H^*(BG;\Q)$-algebras

\begin{equation}
H^*_G(X;\Q) \cong H^*_T(X;\Q)^W
\end{equation}
where $W = N_G(T)/T$ is the Weyl group of $T$ in $G$ \cite[Theorem.2.2]{leray1950homologie}.

We generalize this situation by introducing the following definition motivated by \cite{baird2018cohomology}.
\begin{definition}\label{def2.1}
	Let $G$ be a compact Lie group and $K$ be a closed subgroup such that $G/K$ is path-connected. The pair $(G,K)$ is a  \textit{free extension pair} over a field $\bk$ if the map $H^*(BK;\bk) \rightarrow H^*(G/K;\bk)$ is surjective. 
\end{definition} 
Observe that this condition implies that the action of $G$ on the cohomology of $G/K$ is trivial. Moreover, $H^*(BK; \bk)$ becomes a finitely generated free $H^*(BG;\bk)$-module \cite[Theorem III.4.4]{mimura}. Consequently, we get the following result.

\begin{proposition}\label{prop2.1}
	Let $(G,K)$ be a free extension pair over $\bk$ and $X$ be a $G$-space such that $G$ acts trivially on the cohomology of $X$. There is a natural isomorphism of $H^*(BK)$-modules.
 
	\begin{equation}\label{eq:fep}
	\ec[K]{X} \cong \ec{X} \otimes_{\cs} \cs[K]
	\end{equation}
 where $X$ is a $K$-space by restriction of the $G$-action.
\end{proposition}

\begin{proof}
	The Borel constructions $X_K$ and $X_G$ sit in a pullback diagram
 
	\[
	\begin{tikzcd}
	X_K \rar\dar& X_G \dar \\
	BK  \rar& BG
	\end{tikzcd}.
	\]
 
    There is an Eilenberg-Moore spectral sequence converging to the cohomology of $X_K$ with $E_2$-term given by 
 
    \[  E_2^{*,*} = \text{Tor}_{H^*(BG)}(H^*(BK), H^*_G(X)).
    \]
    As $(G,K)$ is a free extension pair, $H^*(BK)$ is a free $H^*(BG)$-module and thus the higher Tor term vanish. Therefore, the spectal sequence collapses at the zeroth column and thus the isomorphism 
    (\ref{eq:fep}) follows. The naturality of the isomorphism is a consequence of the naturality of the spectral sequences.\qedhere
\end{proof}
This result allows us to describe the syzygies in $G$-equivariant cohomology in terms of the $K$-equivariant cohomology analogously to the reduction from non-abelian compact connected Lie group actions to torus actions \cite[Prop.4.2]{franz2016syzygies} as we state in the following result.

\begin{proposition}\label{prop2.4}
	Let $(G,K)$ be a free extension pair and $X$ be a $G$-space. For any $j \geq 1$, $\ec{X}$ is a $j$-th syzygy over $\cs$ if and only if $\ec[K]{X}$ is a $j$-th syzygy over $\cs[K]$.
\end{proposition}
\begin{proof}
	It is a consequence of the following algebraic fact: Let $R,S$ be rings such that $S$ is a free finitely generated $R$-module. Let $A$ be an $S$-algebra and $B$ and $R$-algebra such that $A \cong B \otimes_R S$ as
	$S$-modules. Then $A$ is a $j$-th syzygy over $S$ if and only if $B$ is a $j$-th syzygy over $R$. The result  follows by combining this fact, the remark after Definition \ref{def2.1} and Proposition \ref{prop2.1}.
\end{proof}

These are examples of free extension pairs besides the cases when $G$ is a compact connected Lie group and $K$ is a its maximal torus,

\begin{example}
	\
	\normalfont
	
	$\bullet$ When $G$ is a torus and $K$ is the maximal elementary abelian $p$-subgroup of $K$.
	
	
	$\bullet$. Let $p=2$. When $G = O(n), SO(n), U(n)$ or $SU(n)$ and $K$ is a maximal elementary abelian $p$-subgroup of $G$.
	
	
\end{example}

Let $\mathbb{F} = \R, \C$ or $\mathbb{H}$. The group of unitary matrices  $U(n+1;\F)$ acts transitively on $\F^{n+1}$ and hence on $\F P^{n}$ (here unconventionally we intend $U(n;\R) = O(n)$). For any point $x \in \F P^{n}$, the isotropy group is homeomorphic to $U(n;\F) \times U(1;\F)$. In the case of $\F = \R$ or $\C$ and $SU(n+1;\F)$ acting on it, the isotropy group is given by $U(n;\F)$. With this remark, we state and prove the following result.

\begin{proposition}\label{prop_matrix}
	Let $n \geq 1$ and $(G_n,K_n) = (SO(n), O(n)), (SU(n), U(n))$ or $(U(n;\F), U(n;\F)\times U(1;\F))$.
	There is an embedding of $K_n$ into $G_{n+1}$ such that $(G_{n+1}, K_n)$ is a free extension pair over $\F_2$ in the first case and over an arbitrary field in the other two cases.
\end{proposition}
\begin{proof}
	Let $\Sigma_n = \R P^n, \C P^n$ or $\mathbb{H} P^n$. There is a transitive action of $G_{n+1}$ on $\Sigma_n$ and thus the equivariant cohomology $H^*_{G_{n+1}}(\Sigma_n)$ is isomorphic to $H^*(B(G_{n+1})_x)$ for any $x \in \Sigma_n$. Using homogeneous coordinates on $\Sigma_n$, we see that the isotropy group of $x = [0:\cdots : 0:1]$ is given by  $(G_{n+1})_x \cong K_n$ and thus the inclusion map $(G_{n+1})_x \rightarrow G_{n+1}$ induces an embedding of $K_n$ into $G_{n+1}$ and so $\Sigma_n \cong G_{n+1}/K_n$ by the orbit-stabilizer Theorem. Furthermore, since $H^*_{G_{n+1}}(\Sigma_n) \cong H^*(BK_n)$, the restriction map $H^*_{G_{n+1}}(\Sigma_n) \rightarrow H^*(\Sigma_n)$ is surjective as it coincides with the polynomial restriction  $\F[x_1,\ldots, x_n] \rightarrow \F[x_1]/(x_1)^{n+1} $.  This shows that $\Sigma_n$ is $G_{n+1}$-equivariantly formal. Combining both facts we have that $(G_{n+1}, K_n)$ is a free extension pair.
\end{proof}
\begin{remark}
Let $G,K$ be topological groups, let $\alpha\colon G \rightarrow K$ be a group homomorphism and let $X$ be a $K$-space. Then $X$ can be made into a $G$-space via the map $\alpha$ and there is an induced map in cohomology $\alpha^*: H^*_K(X) \rightarrow H^*_G(X)$. This map is induced by the equivariant map under the induced action on $EG \times X$ since $f(g(z,x))= \alpha(g) f(z,x)$.
\end{remark}
Recall that for $K \subseteq G$  a closed subgroup of a Lie group $G$, the Weyl group of $K$ in $G$ is defined as $W = N_G(K)/C_G(K)$ where $N_G(K)$ denotes the normalizer of $K$  and $C_G(K)$ the centralizer of $K$ in $G$.
\begin{theorem}\label{weylequivariant}
	Let $(G,K)$ be a free extension pair over a field $\bk$. Let $W$ be the Weyl group of $K$ in $G$ and suppose that the group inclusion of $K$ induces an algebra isomorphism $H^*(BK)^W \cong H^*(BG)$. Then for any  $G$-space $X$ such that $H^*(X)^G = H^*(X)$, the group $W$ acts on the $K$-equivariant cohomology of $X$ and there is a natural isomorphism of $H^*(BG)$-modules $H^*_K(X)^W \cong H^*_G(X)$.
\end{theorem}
\begin{proof}
	Fix an element $\sigma \in G$ and let  $\alpha = c_{\sigma}$ be the conjugation of this element on $G$.  By using Milnor's join construction of $BG$, it can be seen that this map induces the identity map in the cohomology of $H^*(BG)$ \cite[Ch II Thm 1.9]{milgramadem}. This function induces a map $\alpha_G$ on the homotopy quotient $EG \times_G X$ as $\alpha_G([z,x])= [E\alpha(z),x]$.  By Naturality of the Serre spectral sequence for the fibration $X \rightarrow X_G \rightarrow BG$, we get that the map $\alpha_G^*: H^*_G(X)\rightarrow H^*_G(X)$ is determined by the functions $B\alpha^*$ over $H^*(BG)$ and the identity  over the fiber and so $\alpha$ induces the identity on the $G$-equivariant cohomology of $X$.
 
Let $w \in W$. The conjugation $c_w$ induces  a well defined action of $W$ on both $H^*(BK)$ and $H^*_K(X)$ so that the canonical map $H^*_G(X) \otimes_{H^*(BG)} H^*(BK) \rightarrow H^*_K(X)$ of algebras is $W$-equivariant. Moreover, it is  an isomorphism by Proposition \ref{prop2.1}. Since the argument in the previous paragraph shows that $W$ acts trivially on $H^*_G(X)$, and using that $H^*(BK)^W \cong H^*(BG)$ by assumption by assumption,  there is an isomorphism $H^*_G(X) \cong H^*_K(X)^W$ as $H^*(BG)$ by restriction of the above isomorphism to fix points. In particular, we can identify $H^*_G(X) \subseteq H^*_K(X)^W$.  
\end{proof}

Observe that Proposition \ref{prop2.1}  and Theorem \ref{weylequivariant} can be summarized in Theorem \ref{teo1.1} as discussed at the beginning of this document. In the rest of this section, we discuss the case when $G$ is a semidirect product. We first start with the following result.
\begin{proposition}\label{lema2.1}
	Let $G$ and $K$ be groups. Suppose that there is a subgroup $N \subseteq G$ and a group homomorphism $\phi\colon K \rightarrow Aut(G)$ such that $\phi_k\vert_N$ is the identity for all $k \in K$. Then $(G,N)$ is a free extension pair if and only if $(G \rtimes_\phi K, N\times K)$ is a free extension pair. 
\end{proposition}

\begin{proof}
	Under these assumptions, there are canonical isomorphisms 
	\( H^*(B(N \times K))\cong H^*(BN) \otimes H^*(BK) \)   and \( H^*((G \rtimes_\phi K)/(N \times K)) \cong H^*(G/N) \)
	that fit in a commutative diagram
 
	\[
	\begin{tikzcd}
	H^*(BN) \otimes H^*(BK) \dar\rar &H^*((G\rtimes_\phi K)/(N \times K)) \dar \\
	H^*(BN) \rar & H^*(G/N) 
	\end{tikzcd}
	\]
 
	The vertical left arrow is induced by the inclusion of $N$ into $N \times K$ and hence it is surjective. As the vertical right arrow is an isomorphism, the top horizontal arrow is surjective if and only if the bottom arrow is surjective.
\end{proof}
It is not difficult to check from the definition of free extension pairs and the fact that the classifying space functor preserves finite products, that the product of two free extension pairs is again a free extension pair.

\begin{remark}\label{lem2.2}
	Let $G$ and $K$ be groups and $N \subseteq G$ a closed subgroup. Assume that there is a group homomorphism $\phi\colon K \rightarrow Aut(G)$ such that $\phi_k(N) \subseteq N$ for any $k \in K$.  Then the induced semi-direct products $(G\rtimes K, N \rtimes K)$ is a free extension pair if and only if $(G,N)$ is a free extension pair.
\end{remark}

Let $G = N \rtimes T$ be a semidirect product group. The $G$-equivariant cohomology can be computed stepwise as in the direct product case; namely, for a $G$-space $X$,  there is an isomorphism of $\bk$-algebras $H^*_G(X) \cong H^*_K(X_N)$. As a consequence of Proposition \ref{lema2.1} and Remark \ref{lem2.2}, we can recover the free extension property for the matrix groups $G_n$ and $K_n$ of  Proposition  \ref{prop_matrix} since $K_{n} \cong G_n \rtimes K_n/G_n$. Notice that the group $L = K_n/G_n$ is unique up to isomorphism for all $n \geq 1$. Considering the $n$-fold product $L^n$, we conclude the following corollary.
\begin{corollary}
	Let $L = K_n/G_n$. The pairs $(G_{n+1}, L^n)$ and $(K_{n}, L^n)$ are free extension pairs. In particular, for $\bk= \F_2$, if $M$ denotes a maximal elementary abelian $2$-subgroup of $L^n$, the pairs $(G_{n+1},M)$ and $(K_n,M)$ are free extension pairs.
\end{corollary}

\section{Torus actions and compatible involutions}\label{se:toruscomp}

In this section, we will consider cohomology with coefficients over $\bk = \F_2$. Let $X$ be a space with an action of a torus $T$ and let $\tau\colon X \rightarrow X$ be an involution. We say that $\tau$ is  \textit{compatible} if $\tau(g\cdot x) = g^{-1}\cdot \tau(x)$ for any $g \in T$ and $x \in X$. Examples of such spaces appear naturally as toric varieties in algebraic geometry, Hamiltonian torus actions on symplectic manifolds and topological generalizations of these spaces such as  quasitoric manifolds, torus manifolds and moment angle complexes \cite{buchstaber2002torus}, \cite{davis1991convex}, \cite{hattori2001theory}.

Let $G = T \rtimes_\tau \Z/2\Z$ and $K = \langle \tau \rangle = \Z/2\Z$ where $\tau$ acts on $T$ by inversion. The equivariant cohomology of a $T$-space with a compatible involution can be approached by studying the $G$-equivariant cohomology of $X$. We will show that $H^*(BG)$ is a polynomial algebra in $(n+1)$-variables as it will be canonically isomorphic to the cohomology of $B(T\times K)$ as stated in the following result. 

\begin{theorem}\label{teo41}
There is a unique graded algebra isomorphism $\theta: H^*(BG) \rightarrow H^*(BT)\otimes H^*(BK)$ that preserves the canonical maps induced by the inclusion $i$ of $T$ into $G$  and the projection of $G$ onto $K$ so the diagram
	\[
	\begin{tikzcd}
	0\rar & H^*(BK) \arrow[d,equal]\arrow[r,"p^*"] & H^*(BG)\arrow[d,"\theta"]\arrow[r,"i^*"] & H^*(BT)\arrow[d,equal]\rar & 0\\
	0 \rar & H^*(BK)  \arrow[r,"p^*"] & H^*(BT)\otimes H^*(BK)\arrow[r,"i^*"] & H^*(BT) \rar & 0
	\end{tikzcd}
	\]
commutes. In particular, the rows of the diagram are split exact sequences in positive degree cohomology. Moreover, there is a unique  one-to-one section $\varphi: H^*(BT) \rightarrow H^*(BG)$ that satisfies $\coker(\varphi) \cong H^*(BK)^+$. And the map $j^*\colon H^*(BG) \rightarrow H^*(BK)$  induced by the inclusion $j\colon \{e\} \times K \rightarrow G$ has kernel $(H^*(BT)^+)$, and the composite $j^*\circ p^*$ is the identity over $H^*(BK)$.

\end{theorem}
\begin{proof}
	Consider the fibration \(
	BT \rightarrow BG \rightarrow BK
	\).
	Since the action of $\pi_1(BK)$ on $H^*(BT)$ induces the trivial action on the cohomology $H^*(BT;\F_2)$, the Serre spectral sequence associated to this fibration has $E_2$-term
 
	\[E_2 \cong H^*(BK) \otimes H^*(BT) \Rightarrow H^*(BG).\]
 
	By degree reasons and the multiplicative property of the spectral sequence, the only possible non-zero differential $d_3$ is determined by $d_3: E_3^{0,2} \rightarrow E_3^{3,0}$. Choose generators $x_i \in H^2(BT)$  for $1 \leq i \leq n$ and $t \in H^1(BK)$. Under these identifications, we have that $d_3(x_i) = \alpha_i {t}^3$ with either $\alpha_i = 0$ or $\alpha_i = 1$.
	By considering the subextension $1 \rightarrow 1 \rtimes \Z/2 \rightarrow \Z/2$,
	we can check that $d_3 = 0$.
	We have then an isomorphism of $H^*(BK)$-modules 
 
	\[H^*(BK) \otimes H^*(BT) \cong H^*(BG).\]
	
	On the other hand, since $H^*(BT)$ is a finitely generated polynomial algebra, we can choose a multiplicative section ${\varphi}: H^*(BT) \rightarrow H^*(BG)$ of the surjective map $H^*(BG) \rightarrow H^*(BT)$ induced  by the inclusion map. Therefore, such  a map together with the canonical map $p^*\colon H^*(BK) \rightarrow H^*(BG)$ gives rise to an isomorphism of graded $H^*(BK)$-algebras
 
	\[{\theta}: H^*(BK) \otimes H^*(BT) \rightarrow H^*(BG)\]
 given by ${\theta}(\alpha \otimes \beta) = p^*(\alpha){\varphi}(\beta)$ as both $p^*$ and $\varphi$ are multiplicative.
 	
	Note that $j^*\varphi(x_i)$ is either zero or $t^2$ in $H^*(BK)$. We can then make $\varphi$ and hence $\theta$ unique as they are determined  by the generators $x_i$ and the condition that $j^* \varphi(x_i) = 0$ (if the latter is not the case, we just redefine a new $\varphi'$ as $\varphi + \varphi$ over the generators $x_i$ and $\theta$ will be determined by this $\varphi'$). This also implies that the composite \( j^* \theta \colon H^*(BK) \otimes H^*(BT) \rightarrow H^*(BK)\)
	has kernel $H^*(BT)^+$.
	Now notice that the composite 
 
 \[H^*(BT)\xrightarrow{\varphi} H^*(BG) \xrightarrow{i^*} H^*(BT),\] 
 where $i^*$ is induced by the inclusion $T \rightarrow G$, is the identity on $H^*(BT)$ since $i^*(t) = 0$ and $\varphi$ was constructed as a section of this map. This implies that $i^*$ is surjective and $\ker(i^*) \cong (H^*(BK)^+)$. Using a similar argument for the composite $H^*(BK)\xrightarrow{p^*} H^*(BG) \xrightarrow{j^*} H^*(BK)$, which is the identity over $H^*(BK)$, we conclude that the map 
	$p^*$ is surjective and  has cokernel  isomorphic to $H^*(BT)$. 
\end{proof}

Now we will study the algebraic properties of the $G$-equivariant cohomology as a module over $H^*(BG)$. Notice that for any $G$-space $X$, there is an induced involution $\tau$ on the space $X_T$; moreover, the spaces $X_G$ and $(X_T)_\tau$ are homotopically equivalent. Using this remark we prove the following result.

\begin{proposition}\label{teo418}
	Let $X$ be a $G$-space and assume that $X$ is $T$-equivariantly formal. Then $X$ is $G$-equivariantly formal if and only  if the Borel construction $X_T$ is $G/T$-equivariantly formal.
\end{proposition}
\begin{proof}
	Firstly, let us suppose that $X$ is $G$-equivariantly formal. By Theorem \ref{teo41} and the above remark we get isomorphisms
	\begin{align*}
	H_\tau^*(X_T)\cong H^*_G(X) &\cong H^*(BG) \otimes H^*(X) \\
	&\cong H^*(BK)  \otimes H^*(BT) \otimes H^*(X) \\
	&\cong H^*(BK) \otimes H^*(X_T) 
	\end{align*}
	and so $X_T$ is $K$-equivariantly formal. By reversing the above sequence of isomorphisms, the converse of the statement holds. 
\end{proof}

Now we will apply this theorem to the \textit{conjugation spaces} introduced by Haussmann-Holm-Puppe \cite{haussconj}. Examples of such spaces include complex Grassmannians, toric manifolds, polygon spaces and symplectic manifolds with antisymplectic involutions. From \cite[Theorem.7.5]{haussconj}, it follows that $X_T$ is $\tau$-equivariantly formal. As a consequence of Theorem \ref{teo418} we obtain the following result.
\begin{corollary}\label{cor418}
	Let $X$ be a $T$-space which is also a conjugation space with a compatible involution~$\tau$. Then $X$ is $G$-equivariantly formal.
\end{corollary}

\section{Reduction to 2-torus actions}\label{se:2torus}

In this section, we will use the results from Section 2 to study the equivariant cohomology for torus actions and compatible involutions by reducing to the maximal elementary abelian $2$-subgroup (or $2$-torus). For an $n$-dimensional torus $T$, let $T_2$ be the maximal $2$-torus subgroup in $T$ consisting of its elements of order $2$. Thus $K$ acts trivially on $T_2$ and so $H = T_2 \rtimes K = T_2 \times K \cong (\Zd)^n \times (\Zd) \cong (\Zd)^{n+1}$ is a maximal $2$-torus subgroup in $G = T  \rtimes K$. Let us choose generators $H^*(BT)\cong \F_2[x_1,\cdots,x_n]$ and  $H^*(BH) \cong \F_2[t_1,\ldots,t_n,y]$  so that $H^*(BT_2) \cong \F_2[t_1,\ldots, t_n]$ and the map induced by the inclusion $T_2 \rightarrow T$ maps $x_i$ to $t_i^2$ for all $1 \leq i \leq n$. We now compute explicitly the module structure of $H^*(BH)$ over $H^*(BG)$ as stated in the following lemma.
\begin{lemma}\label{lem1.4}
	The map $i^*\colon H^*(BG) \rightarrow H^*(BH)$ induced by the inclusion $i\colon H\rightarrow G$ is given by $i^*(x_i) = t_i^2 + t_iy$ for all $1 \leq i \leq n$ and $i^*(y) = y$.
\end{lemma}
\begin{proof}
We will first prove the result for $n=1$. Let $G = T \rtimes K$ where $T = S^1$ write  and $H^*(BG) \cong H^*(BT) \otimes H^*(BK) = \mathbb{F}_2[x,y]$ as in Theorem \ref{teo41}. We have that $i^*(y)$ is the generator on $H^*(BK)$ in the decomposition $H^*(BH) \cong H^*(BT_2) \otimes H^*(BK)$. So we can write $H^*(BH) \cong \mathbb{F}_2[t,y]$ and $i^*(x) = \alpha t^2 + \beta ty + \gamma y^2$.  The decomposition on Theorem $\ref{teo41}$ implies that $\alpha = 1$ and $\gamma = 0$. To show that $\beta = 1$, consider the Whitney map $H^*(BO(2))\rightarrow H^*(BO(1)) \otimes H^*(BO(1))$ \cite[Theorem.1.2]{Brown} that sends the top class $\omega_2$ to $\omega_1 \otimes \omega_1$ and $\omega_1$ to $\omega_1 \otimes 1 + 1 \otimes \omega_1$.  These maps coincide (up to isomorphism) with the map induced by the inclusion $H \rightarrow G$ since $T \rtimes K \cong O(2)$. By identifying $\omega_2$ with $x$ and $\omega_1$ with $y$ on $H^*(BO(2))$, Theorem \ref{teo41} implies that $\omega_1 \otimes 1$ is sent to $t + y$ and $1 \otimes \omega_1$ is sent to $t$ as $i^*(y)  = i^*(\omega_1 \otimes 1 + 1 \otimes \omega_1)$ must be a generator in $H^*(BK)$ and zero in $H^*(BT_2)$. Therefore,    $i^*(x)$ conicides with $t(t+y) = t^2 +  ty$.
For the general case, let $G = (S^1)^n \rtimes \Zd$ and write $H^*(BG) \cong \F_2[x_1,\ldots, x_n,y]$. For $x_i \in H^2(BG)$, write

\[ i^*(x_i) = \sum_{k=1}^n \alpha_{k}t_k^2 + \sum_{k < l} \beta_{k,l}t_kt_l + \sum_{k=1}^n\gamma_k t_ky \in H^*(BH). \]

The general result follows from the previous case by considering an inclusion of $S^1 \rtimes \Zd$ into one of the $n$ different circle factors of $G$. 
\end{proof}
Theorem \ref{teo41} shows that the cohomology of $H^*(BG)$ is isomorphic to $H^*(B(T\times K))$. However, Lemma \ref{lem1.4} implies that their cohomology as modules over the Steenrod algebra are not isomorphic.

\begin{proposition}
	The mod $2$-cohomology rings of the classifying spaces $BG$ and $B(T\times K)$ are isomorphic as $\F_2$-algebras but not as modules over the Steenrod algebra.
\end{proposition}
\begin{proof}
	As in Lemma \ref{lem1.4}, we may assume $n =1$ first. For $x \in H^2(BG)$ generator associated to the cohomology of $BT$, write $Sq^1(x) = \alpha xy + \beta y^3$ for $\alpha,\beta\in \F_2$. By naturality of Steenrod operations, we have that $i^*(Sq^1(x)) = Sq^1(i^*(x))$ where $i^*$ is the map induced by the inclusion $ H \rightarrow G$. Therefore, $\alpha(t^2y +ty^2) + \beta y^3 = Sq^1(t^2 + ty) = t^2y + yt^2$  by Lemma \ref{lem1.4} and so $\alpha = 1, \beta = 0$. On the other hand, a similar argument applied to the inclusion $j\colon H \rightarrow T \times K $ shows that $Sq^1(x) = 0$ as $j^*(x) = t^2$.
\end{proof}

Since $(S^1, \Zd)$ is a free extension pair over $\F_2$, by Proposition \ref{lema2.1}, Remark \ref{lem2.2} and Theorem \ref{teo41} it follows that $(G,H)$ is a free extension pair as well. Therefore,  as in Proposition \ref{prop2.4},  for any $G$-space $X$, $H^*_G(X)$ is a $j$-th syzygy over $H^*(BG)$ if and only if $H^*_H(X)$ is a $j$-th syzygy over $H^*(BH)$.

Proposition \ref{prop2.1} shows that the $H$-equivariant cohomology of $X$ is determined by the $G$-equivariant cohomology of $X$.  As in the case for compact connected Lie groups and their maximal tori for rational coefficients, we can also describe the $G$-equivariant cohomology of $X$ in terms of the Weyl invariants of the $H$-equivariant cohomology of $X$. Notice that the Weyl group of $H$ in $G$ is  $W= N_G(H)/H$ since $C_G(H) = H$. We first prove the following proposition.
\begin{theorem}\label{weylprop}
	Let $W = N_G(H)/H$ be the Weyl group of $H$ in $G$. Then $W \cong (\Z/2)^n$ and there is an isomorphism of algebras $H^*(BG) \cong H^*(BH)^W$ where the action on the cohomology of $H^*(BH)$ is induced by the conjugation action  of $W$ on $H$.   
\end{theorem}
\begin{proof}
	Write $H = \langle (g_1,e),\ldots, (g_n, e), (1,\tau) \rangle$ where $e \in K$ denotes the identity and $g_i^2$ is an element of order two in $T$. We claim that $N_G(H) \cong (\Z/4)^n \rtimes \Zd$ where $(\Z/4)^n = \langle \theta_1, \ldots, \theta_n \rangle$ is generated by elements $\theta_i$ with $\theta_i^2 = g_i$, and $\Zd$ acts on $\Z/4$ by inversion. Notice that for any $(g, \sigma) \in G$ where $g \in T$ and $\sigma \in \langle \tau \rangle$, the element $(g,\sigma)$ commutes with every element in $H$ of the form $(g_i,e)$ and so we only need to look at the conjugation of the element $(1,\tau) \in H$ by $(g,\sigma)$. Namely, if $(g,\sigma) \in N_G(H)$  we have that  $(g,\sigma)(1,\tau)(g^{\sigma},\sigma) = (g^2, \tau) \in H$ and thus we get  $g \in \langle\theta_1, \ldots, \theta_n \rangle$. This implies that $W \cong (\Zd)^n$ is generated by the cosets $(\theta_i,e)H$ for $i=1,\ldots,n$. 
	
We now compute the induced action on the cohomology of $H^*(BH)$ by the Weyl group $W$. Choose a decomposition $H^*(BH) \cong \mathbb{F}_2[t_1,\ldots, t_n, y]$ where the variables $t_i$ are dual to the generators $g_i$ and $y$ is dual to $\tau$ in $H_1(BH)\cong \F_2[H]$. For a fixed $i \in \{1,\ldots, n\}$, notice that any $(\theta_i,e)H \in W$ acts trivially on the generators $(g_j,e) \in H$; on the other hand, we have that $(\theta_ig_i,e)H \cdot (1,\tau) = (\theta_i,e)(1,\tau)(\theta_ig_i,\tau) = (g_i,\tau)$ This implies that the induced map $\varphi_i$ by the action of $(\theta_i,e)H$ on the cohomology ring $\F_2[t_1,\ldots, t_n,w]$ is given by $\varphi_i(t_j) = t_j$ for $j \neq i$, $\varphi_i(t_i) = t_i + y$ and $\varphi_i(y) = y$. This follows by the duality between homology and cohomology.
	
	Let $\Lambda$ be the set of binary sequences $I =(i_1,\ldots, i_n)$ of length $n$, and write $t^I = t_1^{i_1}\cdots t_n^{i_n}$. The set $\{t^I:I \in \Lambda\}$ form a basis for $H^*(BH)$ as $H^*(BG)$-module. Consider an element $P = \sum_{I \in \Lambda} P_I t^I \in H^*(BH)^W$. Let $I(k)$ be the $k$-th entry of the sequence $I$. We will show that $P_I = 0$ if $I(k) \neq 0$ for some $1\leq k \leq n$.  Let $I \in \Lambda$ be such that $I(k) \neq 0$. Then $\varphi_k(P_It^I) = P_It^I + yP_It^{I_k}$ where $I_k(j) = I(j)$ if $j \neq k$ and $I_k(k) = 0$. Using this notation, we have that $\varphi_k(t^{I_k}) = t^{I_k}$ and  then the equation $P = \varphi_k(P)$ implies that $P_{I_k} + yP_I = P_{I_k}$ and so $P_I = 0$ as desired.
\end{proof}

Actually, the isomorphism of Proposition \ref{weylprop} can be extended to a natural isomorphism in equivariant cohomology as we state in the following consequence of Theorem \ref{weylequivariant}

\begin{corollary}
	Let $X$ be a $G$-space, $H$ the maximal $2$-torus in $G$ and $W$ the Weyl group of $H$ in $G$. Suppose that $G$ acts trivially on the cohomology of $X$. Then there is a natural isomorphism of $H^*(BG)$-algebras 
	\[H^*_G(X) \cong H^*_H(X)^W \]
	induced by the inclusion $H \rightarrow G$.
\end{corollary}

Let $M$ be a symplectic manifold with an action of a torus $T$. Denote by $T_2$ the maximal 2-subtorus of $T$. A consequence of the work of Atiyah \cite{A}  and Frankel \cite{Fr} in equivariant cohomology for Hamiltonian torus actions is that a symplectic action on $M$ is Hamiltonian if and only if $M$ is $T$-equivariantly formal over $\R$. Moreover, if $M$ admits a compatible anti-symplectic compatible involution $\tau$, the \textit{real locus} $M^\tau$ inherits a canonical action of $T_2$ and $M^\tau$ is $T_2$-equivariantly formal over $\F_2$ as shown in \cite{D}, and extended later in \cite{BGH}. This can be summarized in the following result. 

\begin{theorem}\label{TeoremSym}
	Let $M$ be a symplectic manifold with a symplectic action of a torus $T$ and a compatible anti-symplectic involution $\tau$. If $M$ is $T$-equivariantly formal over $\R$, the real locus $M^\tau$ is $T_2$-equivariantly formal over $\F_2$.
	
\end{theorem}

Now we generalize Theorem \ref{TeoremSym} into a topological setting, and over the same coefficient field. Firstly, for a $X$ space $X$ with involution $\tau$, the \textit{real locus} of $X$ is defined as the fixed point subspace $X^\tau$.

Let $G$ be a compact Lie group, $X$ be a $G$-space and $\tau_X$ be an involution on $X$. We say that $\tau_X$ is a \textit{compatible involution} on $X$ if there is a group homomorphism $\tau_G\colon G \rightarrow G$ such that $\tau_G^2 = id$ and $\tau_X(g\cdot x) = \tau_G(g)\cdot\tau_X(x)$ for any $g \in G$ and $x\in X$. The condition of compatibility is equivalent to an action of the group $G_\tau = G \rtimes_\tau \Zd$ on $X$. To simplify our notation, the involutions $\tau_X$ and $\tau_G$ will be both referred as $\tau$, and their domains can be inferred from the context. Notice that the subgroup $G^\tau$ of $\tau$-fixed points of $G$ acts on the real locus $X^\tau$.

\begin{definition}
	Let  $H$ be a $\tau$-invariant closed subgroup of $G$. We say that $(G,H)$ is a \emph{$\tau$-free extension pair} if both $(G,H)$ and $(G^\tau,H^\tau)$ are free extension pairs.
\end{definition}

\begin{theorem}\label{teo4.7}
	Let $G$ be a compact Lie group and let $X$ be a $G$-space with a compatible involution  $\tau$.  Suppose that there is a $\tau$-invariant $2$-torus $H$ in $G$ such that $(G,H)$ is a $\tau$-free extension pair. For any splitting $H_\tau \cong H^\tau \times L$ and   for any integer $j \geq 1$,  if $H^*_{G_\tau}(X)$ is a $j$-th syzygy over $H^*(BG_\tau)$, then so is $H^*_{G^\tau}(X^L)$ as a module over $H^*(BG^\tau)$.
\end{theorem}

\begin{proof}
	As $H$ is a $2$-torus, $(G,H)$ is a free extension pair if and only if $(G_\tau, H_\tau)$ is a free extension pair. In fact, it follows from the commutativity of the diagram
 
	\[ 
	\begin{tikzcd}
	H^*(BH_\tau)\dar \rar & H^*(G_\tau/H_\tau) \dar \\
	H^*(BH) \rar & H^*(G/H)
	\end{tikzcd}
	\]
	where the map $H^*(G_\tau/H_\tau) \rightarrow H^*(G/H)$ is an isomorphism and the canonical map $H^*(BH_\tau) \rightarrow H^*(BH)$ is surjective.
	
	If $H^*_G(X)$ is a $j$-th syzygy over $H^*(BG_\tau)$, it follows from Proposition \ref{prop2.1} that it is also a $j$-th syzygy as a module over $H^*(BH_\tau) \cong H^*(B(H \times \tau))$. We can use now the tools for syzygies for $2$-torus actions discussed in \cite{2tori} Consider a splitting $H_\tau \cong H^\tau \times L$ and identify $L$ with a subgroup of $H_\tau$ via this splitting. From \cite[Prop 2.4]{2tori} applied to the subgroup $L \subseteq H_\tau$, we obtain that $H_{H_\tau/L}^*(X^L) \cong H^*_{H^\tau}(X^L)$ is a $j$-th syzygy over $H^*(B(H_\tau/L)) \cong H^*(BH^\tau)$. Finally, as $(G,H)$ is a  $\tau$-free extension pair, from Proposition \ref{prop2.4} we get that $X^L$ is also a $j$-th syzygy over $H^*(BG^\tau)$.
\end{proof}

Notice that if $\tau$ acts trivially on $H$, the space $X^L$ is the real locus of $X$. 

Theorem \ref{teo4.7} can be applied to the groups $G = T \rtimes (\Zd)^n$ for any $n \geq 0$ which generalize torus actions and torus actions with compatible involutions where $H$ is the maximal $2$-torus in $G$. It also applies to $SO(n)$ with the canonical $\tau$-action that makes the isomorphism $SO(n) \rtimes_\tau \Zd \cong O(n)$ hold. In this case, $H$ is the maximal $2$-torus in $SO(n)$. In particular, we have a generalization of Theorem \ref{TeoremSym} given by the following result.

\begin{theorem}\label{TeoA}
	Let $G = T \rtimes \Zd$ and $X$ be a  $G$-space. If $H^*_G(X)$ is a $j$-th syzygy over $H^*(BG)$, then so is $H^*_{T_2}(X^\tau)$ as a module over $H^*(BT_2)$. In particular, if $X$ is $G$-equivariantly formal, then the real locus $X^\tau$ is $T_2$-equivariantly formal.
\end{theorem}

\begin{example}
	\normalfont
	Let $X$ be a $T$-space. Suppose $X$ is also a conjugation space with a compatible conjugation $\tau$. Then from Theorem \ref{TeoA} and Corollary \ref{cor418} we have that the real locus $X^\tau$ is $T_2$-equivariantly formal.
\end{example}

The assumptions of Theorem \ref{TeoA} cannot be weakened. For example,  If $X$ is a $G$-space such that it is simultaneously $T$-equivariantly formal and $\tau$-equivariantly formal, it is not necessarily true that $X$ is $G$-equivariantly formal or that its real locus $X^\tau$ is $T_2$-equivariantly formal as the next example shows.

\begin{example}
	\normalfont
	
	Let $X = \{ (u, z) \in \C \times \R \mid |u|^2 + |z|^2 = 1 \} = S^{2}$, let $T = S^1$  act on $X$ by $g \cdot (u, z) = (gu, z)$; more precisely, by scalar multiplication in the first factor. Let $\tau$ be the involution $\tau(u,z) = (\bar{u},-z)$ which is compatible with the torus action. Notice that $X^T = \{(0,1), (0,-1)\} \cong S^{0} $ and $X^\tau = \{ (-1,0), (1,0)\} \cong S^{0}$.  Therefore, the action of $T_2$ on  $X^\tau$ is the multiplication by $\pm 1$ and thus it is a free $T_2$-space. This implies that its $T_2$-equivariant cohomology is not free over $H^*(BT_2)$. On the other hand,  $H^*_T(X)$ is a free $H^*(BT)$-module since $\sum_{i\geq 0}\dim H^i(X) = \sum_{i \geq 0} H^i(X^T)$ \cite[Thm 3.10.4]{AP1993}
\end{example}

\subsection*{Acknowledgements}
I would like to thank Matthias Franz for his collaboration and helpful discussions to develop this project. I am also grateful to Jeffrey Carlson for his fruitful  discussions. Finally, I want to thank the anonymous referees for their comments and feedback on earlier versions of this document.

\bibliography{sn-bibliography}

\end{document}